\documentclass[12pt]{amsart}
\usepackage{amssymb, latexsym, amsthm}
\usepackage{color}

\newcommand\dist{{\operatorname{dist}}}
\newcommand\G{{\mathcal{G}}}
\newcommand\suchthat{{\,:\ \,}}

\newcommand{\card}[1]{{\left|{#1}\right|}}
\newcommand{\norm}[1]{{\vert\vert {#1} \vert \vert}}
\newcommand\isom{{\,\cong\,}}

\DeclareMathOperator{\Ker}{Ker} %
\newcommand\ideal[1]{{\left<{#1}\right>}}
\newcommand\sg[1]{{\ideal{#1}}}
\def\sub{\subseteq}

\def\F{{\mathbb{F}}}

\def\co{{\,{:}\,}}
\def\s{{\sigma}}
\newcommand{\set}[1]{{\left\{#1\right\}}}

\def\ra{{\rightarrow}}
\def\hra{{\hookrightarrow}}
\def\lra{{\,\longrightarrow\,}}

\def\minusset{{-}}

\renewcommand{\Im}{\operatorname{Im}} %
\renewcommand{\Ker}{\operatorname{Ker}} %
\newcommand\dimcol[2]{{[{#1}\!:\!{#2}]}} %
\newcommand\PGL[1][d]{{\operatorname{PGL}_{#1}}} %
\def\B{{\mathcal B}}

\newcommand\cat[1]{{\mathbf{#1}}}
\newcommand\obj[1][]{{\operatorname{obj}\if!#1!\else(\cat{#1})\fi}}
\newcommand\Obj[1][]{{\operatorname{obj}\if!#1!\else(\cat{#1})\fi}}

\renewcommand\th[1]{{${#1}^{\rm{th}}$}}

\newcommand\eq[1]{{(\ref{#1})}}%

\long\def\second#1\seconded{#1}
\newcommand\Cvec{{C^1(\vec{X})}}
\newcommand\CvecI{{C^1(\vec{X})'}}
\newcommand\CvecII{{C^1(\vec{X})''}}
\newcommand\Fun[2][X]{{F^{#2}(#1)}}
\newcommand\cells[2][X]{{{#1}_{#2}}} 
\newcommand\frst{{$1^{\rm{st}}$}}
\newcommand\scnd{{$2^{\rm{nd}}$}}

\renewcommand\th[1]{{${#1}^{\rm{th}}$}}

\newtheorem{thm}{Theorem}[section] %

\newtheorem{cor}[thm]{Corollary}
\newtheorem{defn}[thm]{Definition}

\newtheorem{exmpl}[thm]{Example}
\newtheorem{lem}[thm]{Lemma}
\newtheorem{prop}[thm]{Proposition}

\newtheorem{rem}[thm]{Remark}

\newcommand\Cref[1]{{Corollary~\ref{#1}}}
\newcommand\Dref[1]{{Definition~\ref{#1}}}

\newcommand\Eref[1]{{Example~\ref{#1}}}
\newcommand\Lref[1]{{Lemma~\ref{#1}}}
\newcommand\Pref[1]{{Proposition~\ref{#1}}}
\newcommand\Rref[1]{{Remark~\ref{#1}}}
\newcommand\Tref[1]{{Theorem~\ref{#1}}}
\newcommand\Trefs[2]{{Theorems~\ref{#1} and~\ref{#2}}}
\newcommand\Crefs[2]{{Corollaries~\ref{#1} and~\ref{#2}}}
\newcommand\Fref[1]{{Figure~\ref{#1}}}
\newcommand\Sref[1]{{Section~\ref{#1}}}
\newcommand\Srefs[2]{{Sections~\ref{#1} and~\ref{#2}}}
\newcommand\Ssref[1]{{Subsection~\ref{#1}}}

\newcommand\one[1][]{{{\bf{1}}_{#1}}}
\renewcommand\Pr[1]{{\operatorname{Pr}\set{#1}}}
\newcommand\vdel{{\vec{\delta}}}

\newcommand\Cube[8]{
\xymatrix@C=8pt@R=8pt{
{#1} \ar@{-}[rrr] \ar@{-}[ddd] \ar@{-}[rd] & {} & {} & {#2}\ar@{-}[ddd] \ar@{-}[ld] \\
{} & {#5} \ar@{-}[r] \ar@{-}[d] & {#6} \ar@{-}[d]& {} \\
{} & {#8} \ar@{-}[r] \ar@{-}[ld] & {#7} \ar@{-}[rd]& {} \\
{#4} \ar@{-}[rrr] & {} & {} & {#3}
}}

\newcommand\CUBEb[9]{%
    \def\tempa{#1}%
    \def\tempb{#2}%
    \def\tempc{#3}%
    \def\tempd{#4}%
    \def\tempe{\boldsymbol{#5}}%
    \def\tempf{#6}%
    \def\tempg{\boldsymbol{#7}}%
    \def\temph{#8}%
    \CUBEcontinued#9
}
\newcommand\CUBE[9]{%
    \def\tempa{#1}%
    \def\tempb{#2}%
    \def\tempc{#3}%
    \def\tempd{#4}%
    \def\tempe{#5}%
    \def\tempf{#6}%
    \def\tempg{#7}%
    \def\temph{#8}%
    \CUBEcontinued#9
}
\newcommand\CUBEcontinued[4]{%
\xymatrix@C=-4pt@R=-4pt{
{\tempa} \ar@{-}[rrrrrr] \ar@{-}[dddddd] \ar@{-}[rrdd] & {} & {} & {} & {} & {} & {\tempb} \ar@{-}[dddddd] \ar@{-}[lldd] \\
{} & {} & {} & *+[o][F-]{#1} & {} & {} & {} \\
{} & {} & {\tempe} \ar@{-}[rr] \ar@{-}[dd] & {} &{\tempf} \ar@{-}[dd]& {} & {} \\
{} & *+[o][F-]{#4} & {} & {} & {} & *+[o][F-]{#2} & {} \\
{} & {} & {\temph} \ar@{-}[rr] \ar@{-}[lldd] & {} & {\tempg} \ar@{-}[rrdd] & {} & {} \\
{} & {} & {} & *+[o][F-]{#3} & {} & {} & {} \\
{\tempd} \ar@{-}[rrrrrr] & {} & {} & {} & {} & {} & {\tempc}
}
}

\usepackage{xy}
\xyoption{all}

\begin{document}

\title{Property testing and expansion in cubical complexes}

\def\UVemail{vishne@math.biu.ac.il}

\author{ David Garber}
\address{Department of Applied Mathematics, Holon Institute of Technology, Israel; and (sabbatical) Einstein Institute of Mathematics, Hebrew University of Jerusalem, Israel}
\email{garber@hit.ac.il}
\author{ Uzi Vishne
}
\address{Department of Mathematics, Bar Ilan University, Ramat Gan, Israel.}
\email{\UVemail}
\thanks{UV is partially supported by Israeli Science Foundation grant \#1623/16.}

\email{ }

\thanks{}

\renewcommand{\subjclassname}{%
      \textup{2000} Mathematics Subject Classification}


\date{\today}

\begin{abstract}
We consider expansion and property testing in the language of incidence geometry, covering both simplicial and cubical complexes in any dimension.
We develop a general method for the transition from an explicit description of the cohomology group, which need not be trivial, to a testability proof with linear ratio between errors. The method is demonstrated by testing functions on $2$-cells in cubical complexes to be induced from the edges.
\end{abstract}
\maketitle


\section{Introduction}

Property testing is a key concept in randomized algorithms and algorithms of sublinear complexity~\cite{OG}. The goal of the test is to distinguish members of a set (``property'') from those at positive fractional distance from it.

To demonstrate this notion, consider symmetric functions $f \co V \times V \ra \set{1, -1}$ where~$V$ is a finite set. Say that such a function is ``special'' if it has the form $f_{ij} = \alpha_i \alpha_j$ for $\alpha \co V \ra \set{1, -1}$.
To efficiently test~$f$ for being special, one verifies that $f_{ij}f_{jk}f_{ki} = 1$ for random indices $i,j,k$. A special function will always pass the test. It is also the case that
if the probability of success is close to~$1$, then $f$ can be well-approximated by some special function.

This example is given in \cite{LK}, where the authors made the significant observation that expansion in simplicial complexes (introduced in \cite{LM} and \cite{Gromov}) is a form of property testing. Indeed, the product along edges of the triangle $\set{i,j,k}$ is an entry of the differential $\delta^1\!f$ associated to the complete simplicial complex, and such entries are computed in constant time.

A somewhat weaker property, that a symmetric function has the form $f_{ij} = \pm \alpha_i \alpha_j$ for a fixed sign, is tested by the product along the square, $f_{ij}f_{jk}f_{k\ell}f_{\ell i} = 1$, see~\cite{Dinur}. Since this is an entry of the cubical differential~$\delta^1f$, one is led to study expansion in cubical complexes. 
We re-prove this result in \Sref{sec:testB1}, to help the reader follow our main application \Sref{sec:testB2}, which is an analogous result in higher dimension:  testing functions defined on squares for being approximated by functions defined on edges, by taking the product along the faces of a cube (see \Sref{sec:3}).

\smallskip

The main contributions of the paper are:
\begin{enumerate}
\item We cast property testing and expansion into the general framework of cohomology on incidence geometry. This covers expansion in simplicial or cubical complexes, in any dimension (Subsection~\ref{ss:34}).
\item Expansion is a form of property testing (\Tref{LK}).
\item Computing the first and second cohomology of the complete cubical complex (Sections~\ref{sec:4} and \ref{sec:5}), by a delicate analysis of non-symmetric functions on the edges.
\item Testing functions on squares to be defined by edges (Sections~\ref{sec:3} and~\ref{sec:testB2}).
\item A general technique to bound the expansion constant in a cohomological setting, which is necessary when the cohomology is nonvanishing.
(\Sref{sec:prel}).
\item Outline of a proof for testability which should deal with the analogous statements in any dimension (\Sref{sec:9}).
\end{enumerate}

Section~\ref{s:cc} briefly introduces cubical complexes, on which our two examples are based.

\medskip
We thank Roy Meshulam, Ilan Newman, %
and an anonymous referee for their helpful comments.

\section{Testing functions on $2$-cells}\label{sec:3}

This section describes our main application %
in simple terms.
Let~$V$ be a finite set. 
We consider functions from $V^4 = V\times V\times V\times V$ to $\mu_2 = \set{1, -1}$.
Can such a function~$g$ be written in the form
\begin{equation}\label{special2}
g_{ijk\ell} = \pm f_{ij}f_{kj}f_{k\ell}f_{i \ell},
\end{equation}
where $f \co V \times V \ra \mu_2$ (in this order of the indices)?

The symmetric group $S_4$ acts on the set of functions $V^4 \ra \mu_2$ by permuting the indices.
An obvious necessary condition for \eq{special2} is that~$g$ be symmetric under the subgroup $\sg{(13),(24)}$. A necessary condition for \eq{special2} to hold for a {\it{symmetric}} function~$f$ is that~$g$ is symmetric under the action of the dihedral group $D_4$ on the indices. Namely, such functions are defined on {\it{squares}} over~$V$.
It is not hard to see that if $g$ is defined on squares, and has the form~\eq{special2}, then the product of the values of $g$ over the faces of any cube is~$1$. We call this the {\bf{cube condition}}. In \Cref{Z2B2} we show that every function~$g$ satisfying the cube condition (for all cubes) is of the form \eq{special2} for some~$f$ (not necessarily symmetric).

\smallskip
A probabilistic analog follows, showing that the cube condition tests a function $g$ defined on squares for being of the form~\eq{special2}:
\begin{thm}\label{maincomb}
There is a constant $\omega > 0$ such that if a function $g$ on the squares fails the cube condition with probability at most~$p$, then~$g$ can be approximated by a function of the form~\eq{special2}, with an error rate of at most~$\omega^{-1}p$.
\end{thm}

Insisting on $f$ being symmetric poses a problem, because not every function on squares satisfying the cube condition is of the form~\eq{special2} with~$f$ symmetric. However, in~\Ssref{ss:[-1]} we define a function $[-1]$ (which equals $-1$ for exactly~$\frac{2}{3}$ of the squares), and then we have:
\begin{thm}\label{maincomb2}
There is a constant $\omega>0$ such that if a function $g$ on the squares fails the cube condition with probability at most~$p$, then~$g$ or $[-1]g$ can be approximated by a function of the form~\eq{special2} with a symmetric~$f$, with an error rate of at most~$\omega^{-1}p$.
\end{thm}
Namely, the cube condition on $g$ tests for the property that $g$ or $[-1]g$ are of the
form~\eq{special2} with $f$ symmetric.

\smallskip

A more precise formulation of \Trefs{maincomb}{maincomb2} is given in \Crefs{2.1forreal}{2.2forreal}. The proof is based on the description of the space of ``directed boundaries''~$B^2(\vec{X})$ %
in \Sref{sec:5}, and follows from \Cref{finally1}.

\section{Expansion and property testing}\label{sec:2}

After a brief introduction to testing and to incidence geometry (see \cite{Ueberberg}), we phrase the notions of expansion and property testing in the language of incidence geometry. This will naturally lead to the observation that a lower bound on the expansion constant proves testability with linear ratio of errors.

\subsection{Testing as an algorithm}

We briefly recall the standard definition of a testable property, with minor adjustments. For a set~$X$, let~$C(X)$ denote the space of functions $X \ra \mu_2 = \set{1,-1}$.
The normalized Hamming distance is defined by $\dist(f,g) = \norm{fg}$ where
$$\norm{f} = \Pr{f(x) \neq 1}$$
for $x \in X$ chosen uniformly at random.

Let $X$ be a set. An {\bf{$\epsilon$-test}} for a subset $P \sub C(X)$ is a randomized algorithm with a constant number~$q$ of queries whose input is $f \in C(X)$ and whose output is YES with probability $\geq 2/3$ if $f \in P$, and NO with probability $\geq 2/3$ if $\dist(f,P) \geq \epsilon$. The set $P$ is {\bf{testable}} if it has an $\epsilon$-test for every $\epsilon>0$. Here $\dist(f,P) = \min_{p \in P}\dist(f,p)$.

We are more interested in one-sided tests. A {\bf{one-sided $(\epsilon,\eta)$-test}} is a randomized algorithm with~$q$ queries whose input is $f \in C(X)$ and whose output is YES with probability $1$ if $f \in P$, and NO with probability $\geq \eta$ if $\dist(f,P) \geq \epsilon$.

We consider one-sided tests obtained from the following scheme. Let~$Y$ be a set, and $\delta \co C(X)\ra C(Y)$ a function such that each entry of $\delta f$ is a product of a bounded number $q$ of entries of~$f$ (\Dref{firstformaldef} elaborates on this idea).
\begin{rem}\label{howtotest}
Let $P = \Ker(\delta)$. Suppose $\dist(f,P) \geq \epsilon$ implies $\norm{\delta f} \geq \eta$.
Then, verification that $(\delta f)_y = 1$ for a random $y \in Y$ is a one-sided $(\epsilon,\eta)$-test for~$P$.
\end{rem}
The expansion version of a test follows:
\begin{rem}
If $\norm{\delta f} \geq \omega \cdot \dist(f,P)$, then
verification that $(\delta f)_y = 1$ for a random $y \in Y$ is a one-sided $(\epsilon,\omega \epsilon)$-test for~$P$.
\end{rem}

\subsection{Incidence geometries}

The incidence geometry introduced here is used to place $X$, $Y$ and $\delta$ from the previous subsection in a unified framework.

A {\bf{pre-geometry}} is a set of elements with prescribed types, with an {\bf{incidence relation}} $\preceq$ which is a reflexive and symmetric ({\it{sic}}) binary relation such that distinct elements of the same type are not incident in each other. As usual, $x \prec y$ is a shorthand for $x \preceq y$ and $x \neq y$.
A set of elements incident in each other is a {\bf{flag}}. A {\bf{geometry}} is a pre-geometry in which every flag is contained in a flag with one element of every type.

Let $\G$ be a pre-geometry with three types, say $0$, $1$ and $2$. Write $\G = \G^0 \cup \G^1 \cup \G^2$, where $\G^i$ is the set of elements of type~$i$. We say that~$\G$ is {\bf{even}} if for every $x \in \G^0$ and $z \in \G^2$, the number of $y \in \G^1$ for which $x \prec y \prec z$ is even, and (somewhat diverging from standard terminology) {\bf{thin}} if this number is always~$0$ or $2$. Every thin pre-geometry is even.
\begin{exmpl}\label{simple}
Let $(\G,\prec)$ be a pre-geometry with three types. Let $\prec_{ij}$ denote the restriction of $\prec$ to $\G^i \times \G^j$. If $\prec_{01}$ and $\prec_{12}$ have full domain and range, and $\prec_{02}\,=\, \prec_{12} \circ \prec_{01}$,  then $\G$ is a geometry.
\end{exmpl}

We denote $\mu_2 = \set{1,-1}$. As usual, $C^i(\G,\mu_2) = C^i(\G)$ is the space of {\bf{cochains}}, namely functions $\G^i \ra \mu_2$, which is a group under pointwise multiplication. The constant functions, in any type, will be denoted~$\one$ and $-\one$.~For $i = 0,1$ we define the differentials $\delta^i \co C^i(\G) \ra C^{i+1}(\G)$ by
$$(\delta^i f)(y) = \prod_{x \prec y} f(x)$$ for every $y \in \G^{i+1}$, where the product is over $x \in \G^i$ such that $x \prec y$.

Let $Z^1(\G) \sub C^1(\G)$ be the kernel of $\delta^1$; elements of $Z^1(\G)$ are usually called {\bf{cocycles}}. Let $B^1(\G) \sub C^1(\G)$ be the image of~$\delta^0$; these are the {\bf{coboundaries}}. For $f \in C^0(\G)$ we have that $(\delta^1\delta^0f)(z) = \prod_{x \prec y \prec z} f(x)$, so if~$\G$ is even then~$\delta^1 \delta^0 = 0$, and than $B^1(\G) \sub Z^1(\G)$. The {\bf{cohomology group}} of an even geometry~$\G$ is the quotient group $H^1(\G) = Z^1(\G)/B^1(\G)$.

\begin{exmpl}\label{ig}
Let~$X$ be a simplicial complex.
For a fixed $d \geq 0$,
the {\bf{\th{d} incidence geometry of $X$}} is
the geometry $\G$ in which
$\G^i$ is the set of $(d-1+i)$-cells of $X$ ($i=0,1,2$), with (symmetrized) inclusion as the incidence relation. This is a thin geometry.
The cohomology~$H^1(\G)$ is then the simplicial cohomology group $H^d(X)$.

Taking~$X$ to be a cubical complex (see \Sref{s:cc} below) works just as well.
\end{exmpl}

Recall that for $f \in C^{i}(\G)$, we denote
$\norm{f} = \Pr{f(x) \neq 1}$ where $x \in \G^i$ is chosen uniformly at random.
\begin{defn}\label{simdef}

For $f \in C^1(\G)$, we denote the coset $[f] = f \cdot B^1(\G)$ and
$$\norm{[f]} = \min_{f' \in [f]}{\norm{f'}}.$$
\end{defn}

The {\bf{degree}} of $z \in \G^2$ is the number of $y \in \G^1$ incident to~$z$. Most often, $\G$ represents an infinite series of geometries, and is not a fixed object. We say that $\G$ is {\bf{bounded}} if there is some fixed $q$ such that $\deg(z) \leq q$ for all $z \in \G^2$ (so, for example, ``the'' complete $2$-dimensional simplicial complex on $n$ vertices, for arbitrary~$n$, ``is'' bounded with $q = 3$ because every triangle has three edges). Under this assumption, the computation of each entry of $\delta^1 g$ requires a bounded number of queries on~$g$.

\smallskip
Since the coefficients are in a field, the short exact sequence
\begin{equation}\label{ses}
\xymatrix{1 \ar@{->}[r] & B^1(\G) \ar@{^(->}[r] & Z^1(\G) \ar@<0pt>@{->}[r]^{\theta} & H^1(\G) 
\ar@{->}[r] &1}
\end{equation}
splits, and $Z^1(\G) \isom B^1(\G) \times H^1(\G)$.  A subspace $H \leq Z^1(\G)$ will be called  {\bf{independent}} if $B^1(\G) \cap H = 0$, equivalently if the restriction $\theta|_{H} \co H \ra H^1(\G)$ is an injection.

\subsection{Testing}

Let $\G$ be a bounded even incidence geometry on the three types $0,1,2$. We specialize \Rref{howtotest} to $\delta^1 \co C^1(\G) \ra C^2(\G)$.
\begin{defn}\label{Testdef}
Let $H \leq Z^1(\G)$ be an independent subspace.
If for every $g \in C^1(\G)$ there is $\alpha \in H$ for which
$$\norm{[g \cdot \alpha]} \leq \omega^{-1}\norm{\delta^1 g}$$
for some (typically small) constant $\omega > 0$, then verification that $(\delta^1 f)_z = 1$ for a random $z \in \G^2$ is a one-sided $(\epsilon,\omega \epsilon)$-test for~$B^1(\G) \cdot H$ for every~$\epsilon$.
When this is the case, we say that
the differential  $\delta^1$ {\bf{tests $B^1(\G)\cdot H$}}, and $B^1(\G)\cdot H$ is {\bf{testable}} (with ratio $\omega$).
\end{defn}
(The condition depends on $H$ only through the product $B^1(\G)\cdot H$.)

In other words, $\delta^1$ tests the space $B^1(\G) \cdot H$ if whenever $\delta^1g$ is nearly~$\one$, the function~$g$ can be ``corrected'' by an element of $H$ so that it is nearly of the form $\delta^0 f$ for some $f \in C^0(\G)$. From an algorithmic perspective, this means that after testing the equality $(\delta^1g)_x = 1$ for a relatively small number of cells $x \in \G^2$, we may conclude that up to $H$, $g$ can be well-approximated in the form $\delta^0 f$, where the quality of the approximation improves as $\omega$ increases. The correction by an element of $H$ is necessary precisely because not every element of $Z^1(\G)$ is of the form $\delta^0f$. For this reason, $\delta^1$ can only test $B^1(\G) \cdot H$ when $H \isom H^1(\G)$.

\begin{rem}\label{trivial8}
Assume $H \isom H^1(\G)$. If $g \in Z^1(\G)$, then $\alpha = \psi \theta(g)$, where $\psi \co H^1(\G) \ra H \sub Z^1(\G)$ splits \eq{ses},  satisfies $\norm{[g \cdot \alpha]} = 0$, so
the requirement in \Dref{Testdef} holds trivially for such~$g$.
\end{rem}

In the case when $H^1(\G) = 0$, $B^1(\G)$ is testable
if
for every $g \in C^1(\G)$ we have that $\norm{[g]} \leq \omega^{-1}\norm{\delta^1 g}$. This is essentially the definition of membership testability in \cite[Defn.~3]{LK}, where we consider the number of errors in the function $\delta^1g$ rather than the probability of a $q$-query algorithm to fail to recognize that $g \not \in B^1(\G)$.

\subsection{The expansion constant}\label{ss:34}

Again let $\G$ be a bounded even incidence geometry on the three types $0,1,2$.
\begin{defn}\label{Expdef}
The {\bf{expansion constant}} of~$\G$ with respect to an independent subspace $H \leq Z^1(\G)$ is
$$\omega_H(\G) = \min_g \max_{\alpha \in H} \frac{\norm{\delta^1g}}{\norm{[g \cdot \alpha]}}$$
where the external minimum is taken over all functions $g \in C^1(\G)$ for which $g \not \in Z^1(\G)$.
\end{defn}
As with simplicial complexes, we say that a family of incidence geometries is a {\bf{family of expanders}} if their expansion constant is bounded away from zero.
Again, when $H^1(\G) = 0$,
$$\omega(\G) = \min_g \frac{\norm{\delta^1g}}{\norm{[g]}}$$
is the coboundary expansion constant as defined in \cite[Defn.~1]{LK} (and the references therein).
On the other hand when $H \isom H^1(\G)$, we obtain the cosystolic expansion constant appearing in \cite{Gromov} (called $\F_2$-cocycle expansion in \cite[Defn.~1.4]{LKK}).
We comment that in this case the expansion constant can also be viewed as the operator norm of the inverse map $(\delta^1)^{-1} \co B^2(\G) \ra C^1(\G)/Z^1(\G)$. 

\medskip

The following result, that expansion implies testability, generalizes \cite[Thm.~8]{LK} (where it is proved for $H = 0$).
\begin{thm}\label{LK}
Let $H \leq Z^1(\G)$ be an independent space as above. Let $\omega_H(\G)$ be the expansion constant of~$\G$ with respect to~$H$. Let $\omega > 0$ be a constant. Then~$\delta^1$ tests the space $B^1(\G) \cdot H$ with ratio~$\omega$, if and only if $\omega \leq \omega_H(\G)$.
\end{thm}
\begin{proof}
By \Rref{trivial8}, $\delta^1$ tests
the space $B^1(\G) \cdot H$ with ratio~$\omega$ if
for every $g \in C^1(\G) \minusset Z^1(\G)$ there is some $\alpha \in H$ for which $\omega \leq \frac{\norm{\delta^1 g}}{\norm{[g \cdot \alpha]}}$.
In other words,
for every $g \in C^1(\G) \minusset Z^1(\G)$, $\omega \leq \max_{\alpha \in H} \frac{\norm{\delta^1 g}}{\norm{[g \cdot \alpha]}}$. But this condition is precisely saying that $\omega \leq \min_g \max_{\alpha \in H} \frac{\norm{\delta^1g}}{\norm{[g \cdot \alpha]}} = \omega_H(\G)$.
\end{proof}

Let us demonstrate the language of incidence geometry by casting the classical linearity test \cite{Blum1993} in this form.
\begin{exmpl}\label{BLR}
Let~$V$ be a vector space over~$\F_2$, of finite dimension $\geq 2$. Blum, Luby and Rubinfeld \cite{Blum1993} showed in their 1993 foundational paper  that a single condition $f(x+y)=f(x)+f(y)$ tests a function $f \co V \ra \F_2$ for linearity:
the probability that a random condition fails is proportional to the distance of~$f$ from the space $\operatorname{Hom}(V,\F_2)$ of linear functions.

Let us construct an incidence geometry $\G$ for which $\delta^1$ realizes this test. As $\G^0$ we take a basis of the dual space of $V^*$. We take $\G^1 = V - \set{0}$, and $\G^2 = \set{\set{a,b,c} \sub V \suchthat a+b+c=0}$. Then $C^1(\G)$ can be identified with functions $f \co V \ra \F_2$ satisfying $f(0) = 0$.
With the notation of \Eref{simple}, we set
$\varphi\!\!\!\;\prec_{01}\!v$ if $\varphi(v) = 1$; $v\!\prec_{12}\!\ell$ if $v \in \ell$; and $\prec_{02}\, =\, \prec_{12} \circ \prec_{01}$.
It follows that $\G = \G^0\cup \G^1\cup \G^2$ is a thin geometry. For $\alpha \in C^0(\G)$ we have that $(\delta^0\alpha)_v = \sum_{\varphi \prec v} \alpha_\varphi = \sum_\varphi \alpha_\varphi\varphi(v)$ so that $\delta^0\alpha = \sum_\varphi \alpha_\varphi \varphi$, which shows that $B^1(\G)$ are precisely the linear functions. It also follows that $H^1(\G) = 0$. Since $(\delta^1f)_\set{a,b,c}  = f(a)+f(b)+f(c)$, $\delta^1$~is the Blum-Luby-Rubinfeld linearity test.
\end{exmpl}

As another example,
consider the $3$-dimensional Bruhat-Tits building $\B = \tilde{A}_3(F)$ associated with $\PGL[4](F)$, where $F$ is a local field. In \cite[Theorem~1.8]{LKK} the authors proved that the family of Ramanujan non-partite quotients of $\B$ is a family of expanders. As a corollary, we now have:
\begin{cor}
Let $X$ be a non-partite Ramanujan quotient of the building $\B$.
Then $\delta^2 \co C^2(X) \ra C^3(X)$ tests the space $Z^2(X)$.
\end{cor}

The expansion constant of the hypercube was computed by Gromov~\cite{Gromov}, also see \cite[Section~4]{KM}.

\section{Cubical complexes}\label{s:cc}

This section briefly presents cubical complexes.
Fix a vertex set~$V$. A~cubical cell of dimension~$d$, or a {\bf{$d$-cell}}, on~$V$ is a subset of $2^d$ vertices in~$V$\!, endowed with the graph structure of the $d$-dimensional cube $\set{0,1}^d$. A subgraph of a cell~$c$ which is itself a cell is called a {\bf{face}} of~$c$. A face of maximal dimension is a {\bf{wall}} of~$c$.
We let $c'\prec c$ denote that~$c'$ is a face of~$c$. A {\bf{cubical complex}} on~$V$ is a collection of cubical cells, of varying dimensions, which includes with a cell all of its faces, and such that every point $i \in V$ is a $0$-cell. The empty set is considered a $(-1)$-cell of the complex. We let $\cells[X]{d}$ denote the collection of $d$-cells in the complex~$X$. The {\bf{dimension}} of $X$ is the largest dimension of a cell.

The cohomology we consider on $X$ is with coefficients in the group $\mu_2 = \set{1,-1}$ of two elements. Let $C^d(X)$ be the functions from $\cells[X]{d}$ to~$\mu_2$. 
The differential map
$$\delta^d \co C^d(X) \ra C^{d+1}(X)$$
is defined by $(\delta^df)_c = \prod_{c'\prec c} f_{c'}$, ranging over the $2d$ walls of~$c$ (there are $4$ walls if $d = 2$, and so on). For example $(\delta^0\alpha)_{ij} = \alpha_i \alpha_j$ for $\alpha \in C^0(X)$.
A~face of co-dimension~$2$ is a wall in exactly two walls, and so $\delta^{d+1}\delta^d = 0$. \Eref{ig} connects this setup to incidence geometry in the obvious manner.

As usual, we set $Z^d(X) = \Ker(\delta^d)$ and $B^d(X) = \Im(\delta^{d-1})$, so that $B^d(X) \sub Z^d(X)$, and the cubical cohomology is the quotient $H^d(X) = Z^d(X)/B^d(X)$.
In any dimension $d \geq 0$, the constant function $-\one \in C^d(X)$ is in fact in $Z^d(X)$, because the $(d+1)$-cube has an even number of faces.

The {\bf{complete cubical complex}} of dimension~$d$ is the cubical complex in which every subset of $2^d$ vertices forms a $d$-cell in all the $(2^d)!/(2^dd!)$ possible ways. In dimension~$1$ this is the complete graph. The complete $2$-dimensional complex on $\set{1,2,3,4}$ has three $2$-cells, corresponding to the enumerations of the vertex set as vertices of a square. 
We compute the first and second cohomology groups of a complete cubical complex in \Srefs{sec:4}{sec:5}, respectively.

\medskip

Although functions on cells are most natural to consider, we will occasionally need functions on arbitrary tuples of vertices.
\begin{defn}\label{Xk}
We denote by $X^{[k]}$ the set of $k$-tuples with distinct entries in the vertex set of $X$, and by $\Fun{k}$ the set of functions $X^{[k]} \ra \mu_2$.
\end{defn}
For example $\Fun{1} = C^0(X)$ and $\Fun{2} = \Cvec$ (see \Ssref{ss:af}), whereas $C^1(X)$ are the symmetric functions $X^{[2]} \ra \mu_2$.
In general $C^d(X) \sub \Fun{2^d}$, with proper inclusion for $d > 0$ due to the symmetry of cells in the left-hand side.

\section{The first cohomology of the complete cubical complex}\label{sec:4}

Let $X$ be the complete cubical complex of dimension~$2$, on at least three vertices.
We define a function $\Delta \co C^1(X) \ra \Fun{3}$ by
\begin{equation}\label{Deltadefall}
(\Delta f)_{ijk} = f_{ij}f_{jk}f_{ki}.
\end{equation}

\begin{lem}\label{bT}
Let $f \in C^1(X)$. Then $f \in Z^1(X)$ if and only if $\Delta f$ is a constant function.
\end{lem}
\begin{proof}
First assume $f_{ij}f_{jk}f_{ki}$ is independent of the triple. For any square $(ijk\ell)$ we have that $$(\delta^1f)_{ijk\ell } = f_{ij}f_{jk}f_{k\ell }f_{\ell i} = (f_{ij}f_{jk}f_{ki})(f_{ik}f_{k\ell }f_{\ell i}) = 1,$$
 so that $f \in Z^1(X)$.

On the other hand, let $f \in Z^1(X)$. Clearly, $(\Delta f)_{ijk}$ does not depend on the order of the indices. For distinct $i,j,k,\ell$ we have that $(\Delta f)_{ijk}(\Delta f)_{jk\ell} = (\delta^1f)_{ij\ell k} = 1$. It follows that if $\card{\set{i,j,k} \cap \set{i',j',k'}} = 2$ then $(\Delta f)_{ijk} = (\Delta f)_{i'j'k'}$; but one can get from a fixed triple to any triple by changing one entry at a time, proving that $\theta = (\Delta f)_{ijk}$ is a constant.
\end{proof}

We can now describe the functions in~$Z^1(X)$.

\begin{thm}\label{C1t}
Let $f \in C^1(X)$. Then $f \in Z^1(X)$ if and only if there are a constant $\theta \in \mu_2$ and a function $\alpha \in C^0(X)$ such that \begin{equation}\label{BE}
f_{ij} = \theta \alpha_i \alpha_j.
\end{equation}
\end{thm}
\begin{proof}
If $f_{ij} = \theta\alpha_i \alpha_j$, then $$(\delta^1f)_{ijk\ell} = f_{ij}f_{jk}f_{k\ell}f_{\ell i} = \theta^4 \alpha_i^2\alpha_j^2\alpha_k^2\alpha_{\ell}^2 = 1$$ for every distinct $i,j,k,\ell$, and so $f \in Z^1(X)$.

Now assume $f \in C^1(X)$ is in the kernel of $\delta^1$. By \Lref{bT}, $\theta = f_{ij}f_{jk}f_{ki}$ is a constant. Fix some vertex $i_0$. Choose $\alpha_{i_0} \in \mu_2$ arbitrarily, and let $\alpha_j = \theta \alpha_{i_0} f_{{i_0}j}$ for every $j \neq {i_0}$. This solves \eq{BE} if ${i_0} \in \set{i,j}$; otherwise, $\theta \alpha_i \alpha_j = \theta^3 \alpha_{i_0}^2 f_{{i_0}i}f_{{i_0}j} = \theta (\Delta f)_{{i_0}ij}f_{ij} = f_{ij}$, as claimed.
\end{proof}

Following \Lref{bT} we define 
\begin{equation}\label{Deltadef}\Delta \co Z^1(X) \ra \mu_2\end{equation}
by $\Delta f = f_{i_0 i_1} f_{i_1i_2} f_{i_2i_0}$, where $i_0,i_1,i_2$ is any triple of distinct vertices. This map is onto, because the constant function $(-\one)_{ij} = -1$ maps to~$-1$.

\begin{prop}\label{KerDelta}
$\Ker(\Delta) = B^1(X)$. More explicitly, in the presentation \eq{BE} we have that $\theta = \Delta f$.
\end{prop}
\begin{proof}
Let $f\in Z^1(X)$. By \Tref{C1t} we may write $f = \theta \cdot \delta^0 \alpha$ for $\alpha \in C^0(X)$. Now for distinct $i_0,i_1,i_2$, $\Delta f = f_{i_0i_1}f_{i_1i_2} f_{i_2i_0} = \theta^3 \alpha_{i_0}^2  \alpha_{i_1}^2\alpha_{i_2}^2= \theta$. Therefore $\Delta f = \one$ if and only if $f \in B^1(X)$.
\end{proof}

\begin{cor}\label{H1=mu2}
The first cohomology of~$X$ is $H^1(X) \isom \mu_2$.
\end{cor}
\begin{proof}
$H^1(X) = Z^1(X)/B^1(X) = Z^1(X)/\Ker(\Delta) \isom \Im(\Delta) = \mu_2$.
\end{proof}

\section{The second cohomology of the complete cubical complex}\label{sec:5}

In this section we consider the complete cubical complex~$X$ of dimension~$3$.
In \Tref{H2=} we prove that $H^2(X) = \mu_2 \times \mu_2$, obtaining along the way a detailed description of key subgroups of $Z^2(X)$.

The description of functions with vanishing $\delta^2$ requires extending $C^d(X)$ to functions which are not necessarily symmetric. Once developed, the same technique characterizes a somewhat more general set of functions, as we will see below.

\subsection{Generalized differentials}\label{ss:gd}

Let $X$ be a cubical complex. For every $d < d'$, let $\delta^{dd'} \co C^d(X) \ra C^{d'}(X)$ be the map defined for $f \in C^d(X)$ by letting $(\delta^{dd'}(f))_c$ be the product of $f(x)$ over the $d$-dimensional faces $x \prec c$.

In particular, $\delta^d = \delta^{d,d+1}$ is the ordinary $d$-dimensional differential.

\begin{rem}
Let $d<d'<d''$. The number of $d'$-cells which are faces of a given $d''$-cell and containing a given $d$-cell is $\binom{d''-d}{d'-d}$. Therefore,
$$\delta^{d'd''} \delta^{dd'} = \binom{d''-d}{d'-d} \delta^{dd''},$$
where $\binom{d''-d}{d'-d}$ is taken modulo $2$.
\end{rem}

In particular, since $\delta^{01} =\delta^0$ and $\delta^{23} = \delta^2$,
\begin{equation}\label{03use}
\delta^{03} = \delta^{13}\delta^0 = \delta^2 \delta^{02}.
\end{equation}

\subsection{Asymmetric functions}\label{ss:af}

This subsection, as well as Subsections~\ref{ss:af+} and \ref{ss:af++}, develop the relations exhibited in \Fref{fig1}.

Let $\Cvec$ denote the space of functions on the {\emph{directed}} underlying graph of $X$, with values in~$\mu_2$. There is a norm function $$N \co \Cvec \ra C^1(X)$$  defined by $(Nf)_{ij} = f_{ij}f_{ji}$. There is also an embedding $C^1(X) \hra \Cvec$, defined by inducing a function from the undirected graph $\cells[X]{1}$ to the directed graph $\cells[\vec{X}]{1}$ by forgetting directions.  Under this embedding, $$C^1(X) = \set{f \in \Cvec \suchthat Nf = \one}.$$
Similarly, we set
\begin{eqnarray*}
\CvecI & = & \set{f \in \Cvec \suchthat Nf \in Z^1(X)},\\
\CvecII & = & \set{f \in \Cvec \suchthat Nf \in B^1(X)};
\end{eqnarray*} 
so that $C^1(X) \sub \CvecII \sub \CvecI \sub \Cvec$.
\begin{figure}
$$\xymatrix@C=16pt@R=12pt{
{} & C^1(X) \ar@{-}[d] & \Cvec \ar@{->}[l]_{N} \ar@{-}[d] & Z^2(X) 
\ar@{-}[d]
\\
\mu_2 \ar@{-}[d] & \ar@{->}[l]_{\Delta} Z^1(X) & \ar@{->}[l]_{N} \CvecI \ar@{->}[r]^{\vdel^1} & B^2(\vec{X}) \ar@{-}[d]
\\
1 & B^1(X) \ar@{-}[u]  \ar@{->}[l]_{\Delta} & \ar@{->}[l]_{N} \CvecII \ar@{-}[u] \ar@{->}[r]^{\vdel^1} & B^2(X) \ar@{=}[d] \\
  & 1 \ar@{-}[u] & \ar@{->}[l]_{N} C^1(X) \ar@{-}[u] \ar@{->}[r]^{\delta^1} & B^2(X) \\
}$$
\caption{Subspaces of $\Cvec$ and $C^2(X)$. Solid lines represent bottom-to-top inclusion; arrows are maps; the double line is equality.}\label{fig1}
\end{figure}

\begin{rem}\label{62}
We may extend $\delta^1 \co C^1(X) \ra C^2(X)$ to a function $\vdel^1$ from $\Cvec$ (to functions on directed $2$-cells), by
\begin{equation}\label{del1def}
(\vdel^1f)_{ii'i''i'''} = f_{ii'}f_{i''i'}f_{i''i'''}f_{ii'''}
\end{equation}
(in this particular order of the arrows, depicting the directed graph~$K_{2,2}$). Under this definition, $\CvecI$ is the space of functions~$f$ for which $\vdel^1f \in C^2(X)$, namely
for which $\vdel^1f$ is symmetric under the action of the dihedral group~$D_4$. Indeed, $\vdel^1f$ is a-priory symmetric with respect to $\sg{(13),(24)} \sub S_4$, so full symmetry is attained when $(\vdel^1f)_{ii'i''i'''} = (\vdel^1f)_{i'i''i'''i}$, but this is equivalent to  $\delta^1(Nf) = \one$, namely $Nf \in Z^1(X)$.
\end{rem}

We thus define $B^2(\vec{X}) = \set{\vdel^1 f \suchthat f \in \CvecI}$.
\begin{prop}\label{B2vecinZ2}
$B^2(X) \sub B^2(\vec{X}) \sub Z^2(X)$.
\end{prop}
\begin{proof}
The left inclusion is obvious because the restriction of $\vdel^1$ to $C^1(X)$ is~$\delta^1$. Let $f \in \CvecI$. In order to prove the right inclusion, we need to verify that $\delta^2\vdel^1f = \one$. Let $c$ be a $3$-cell, namely a cube, whose $1$-skeleton is bipartite. Choose an even-odd partition of the vertices of~$c$, induced by the $1$-skeleton of the cell (000,011,101,110 vs.\,001,010,100 and 111). Direct the edges of $c$ to go from the even to the odd vertices, and present each wall~$s$ of~$c$ as $s=(ii'i''i''')$ where $i$ is even. Now every edge appears in the two faces of $c$ twice in the same direction, so that $(\delta^2\vdel^1f)_c = 1$ by cancelation, regardless of~$f$.
\end{proof}

\subsection{The functions in $\CvecII$}\label{ss:af+}

Define the head and tail functions $\eta_h, \eta_t \co C^0(X) \ra \Cvec$ by $(\eta_h\alpha)_{ij} = \alpha_i$ and $(\eta_t\alpha)_{ij} = \alpha_j$. Note that
\begin{equation}
\label{inB1}
N(\eta_h \alpha) = \eta_h(\alpha) \eta_t(\alpha) = \delta^0\alpha \in B^1(X).
\end{equation}

\begin{prop}\label{64}
We have that $\CvecII = C^1(X)\Im(\eta_h)$.
\end{prop}
\begin{proof}
Let $\alpha \in C^0(X)$. By \eq{inB1} and the definition, $\eta_h \alpha \in \CvecII$. This proves the inclusion~$\supseteq$. On the other hand, if $f \in \CvecII$ then  $Nf = \delta^0\alpha$ for some $\alpha \in C^0(X)$, and then $N(f \cdot \eta_h \alpha) = Nf \cdot \delta^0\alpha = \one$, so that $f \cdot \eta_h \alpha \in C^1(X)$ and $f \in C^1(X)\Im(\eta_h)$.
\end{proof}

Let $Z^1(\vec{X}) = \Ker(\vdel^1) \cap \CvecI = \set{f \in \CvecI \suchthat \vdel^1f = \one}$.
\begin{prop}\label{Z1vec}
We have that $Z^1(\vec{X}) \sub \CvecII$.
\end{prop}
\begin{proof}
Let
$f \in \CvecI$ be such that $\vdel^1f = \one$.
Let $i,j,k$ be distinct vertices. Since $Nf \in Z^1(X)$, $\Delta(Nf) = (Nf)_{ij}(Nf)_{jk}(Nf)_{ki}$. Let $a \neq i,j,k$ be a fourth vertex. We have that
$$
\Delta(Nf) =
(\vdel^1f)_{iajk}(\vdel^1f)_{jaki}(\vdel^1f)_{kaij} = 1,$$
since the edges from $i,j,k$ to $a$ cancel. (This computation is formalized in \Rref{T8}).
By \Pref{KerDelta}, $Nf \in B^1(X)$, and thus $f \in \CvecII$.
\end{proof}

\begin{prop}\label{deltaC1''}
Let $f \in \CvecI$. Then $\vdel^1f \in B^2(X)$ if and only if $f \in \CvecII$.
\end{prop}
\begin{proof}
First assume $f \in \CvecII$. We apply \Pref{64}: Up to an element of $C^1(X)$, whose image under $\delta^1$ is clearly in $B^2(X)$, we may assume $f = \eta_h\alpha$ for $\alpha \in C^0(X)$. Now
 $$(\vdel^1(\eta_h\alpha))_{ijk\ell} = (\eta_h\alpha)_{ij}
(\eta_h\alpha)_{kj}
(\eta_h\alpha)_{k\ell}
(\eta_h\alpha)_{i\ell} = \alpha_i^2 \alpha_k^2 = 1,$$
so that $\vdel^1 f \in B^2(X)$.

Now, if $\vdel^1f \in B^2(X) = \delta^1(C^1(X))$, then by definition there is $g \in C^1(X)$ such that $\vdel^1(fg) = \one$, and $f \in C^1(X) Z^1(\vec{X}) \sub \CvecII$ by \Pref{Z1vec}.
\end{proof}

\subsection{The second differential}\label{ss:af++}

Our goal here is to describe $Z^2(X)$, namely those functions $g \in C^2(X)$ for which $\delta^2 g = \one$. Slightly more generally, we consider functions $g \in C^2(X)$ for which there is $\alpha \in C^0(X)$ such that $\delta^2g = \delta^{03}\alpha$. Explicitly, this condition holds if for every cube, denoting the vertices in a disjoint pair of faces by $[ijk\ell]$ and $[i'j'k'\ell']$, we have that
$$g_{ijj'i'}\ g_{jkk'j'}\ g_{i'j'k'\ell'}\  g_{kk'\ell'\ell}\
g_{ii'\ell'\ell}\ g_{ijk\ell} = \alpha_i \alpha_j \alpha_k \alpha_{\ell} \alpha_{i'} \alpha_{j'} \alpha_{k'} \alpha_{\ell'}.$$
We assume $\card{\cells[X]{0}}\geq 10$, so there are sufficiently many $3$-cells to play with.
\begin{prop}\label{X3}
Let $g \in C^2(X)$. Assume $\delta^2g \in \Im(\delta^{03})$. Then for every distinct $a,b,i,i',j,j'$ we have that
\begin{equation}\label{X3=}
g_{aibj} g_{ajbj'} g_{aj'bi'} g_{ai'bi} = 1.
\end{equation}
\end{prop}
\begin{proof}
Let $s,t,s',t'$ be distinct vertices, disjoint from $a,b,i,i',j,j'$ (this is possible because $\card{X_0}\geq 10$).
Consider the following four $3$-cells, in which identical faces are denoted by the same circled number: 
$$
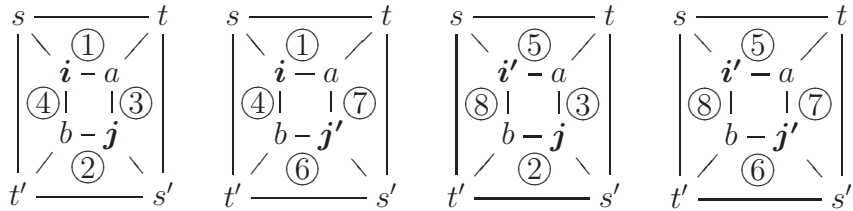
\CUBEb{s}{t}{s'}{t'}{i}{a}{j}{b}{{1}{3}{2}{4}}
\quad
\CUBEb{s}{t}{s'}{t'}{i}{a}{j'}{b}{{1}{7}{6}{4}}
\quad
\CUBEb{s}{t}{s'}{t'}{i'}{a}{j}{b}{{5}{3}{2}{8}}
\quad
\CUBEb{s}{t}{s'}{t'}{i'}{a}{j'}{b}{{5}{7}{6}{8}}
$$
The product of $(\delta^2g)_{c}$ ranging over the four $3$-cells is~$1$ by assumption, because each vertex appears an even number of times. But this product is the left-hand side of \eq{X3=}, because all the other faces, including $[sts't']$, cancel.
\end{proof}

For a subgroup $A \sub C^2(X)$, we let $\pm A$ denote the subgroup $\sg{-\one,A}$ generated by $A$ and the constant function $-\one$.
\begin{thm}\label{goodC2}
We have that
$$Z^2(X)\Im(\delta^{02}) = \pm B^2(\vec{X}) \Im(\delta^{02}).$$
\end{thm}
\begin{proof}
Following \Pref{B2vecinZ2} the inclusion $\supseteq$ is clear because  $-\one \in Z^2(X)$. For vertices $a,b$, let $X^{ab}$ denote the cubical complex obtained from $X$ by removing the vertices $a,b$ and every cell passing through either of them. Let~$g \in Z^2(X)$. Abusing notation, we define $f^{ab} \in C^{1}(X^{ab})$ and $f_{ij} \in C^1(X^{ij})$ by $f^{ab}_{ij} = g_{aibj}$. By \Pref{X3} we have that $\delta^1(f^{ab}) = \one$.
Therefore, by \Tref{C1t}, there are $\theta^{ab} \in \mu_2$ and $\alpha^{ab} \in C^0(X^{ab})$ such that
\begin{equation}\label{Z5}
f^{ab}_{ij} = \theta^{ab} \alpha^{ab}_{i} \alpha^{ab}_j
\end{equation}
for every $i,j$. Since $f^{ab} = f^{ba}$, we may assume $\theta^{ba} = \theta^{ab}$ and $\alpha^{ba} = \alpha^{ab}$ as well. In particular we may view $\theta$ as an element of~$C^1(X)$.
By \Pref{KerDelta}, $\theta^{ab} = \Delta(f^{ab})$, which we may calculate by fixing distinct $i,j,k$ as $f^{ab}_{ij}f^{ab}_{jk}f^{ab}_{ki}$. Since $$f^{ij}_{ab}  = g_{ibja} =
g_{aibj} = f^{ab}_{ij},$$
we have that \begin{eqnarray*}
(\delta^1\theta)_{abcd} & = & \theta^{ab}\theta^{bc}\theta^{cd}\theta^{da} \\
& = & \Delta(f^{ab})\Delta(f^{bc})\Delta(f^{cd})\Delta(f^{da}) \\
& = & (f^{ab}_{ij}f^{ab}_{jk}f^{ab}_{ki})
(f^{bc}_{ij}f^{bc}_{jk}f^{bc}_{ki})
(f^{cd}_{ij}f^{cd}_{jk}f^{cd}_{ki})
(f^{da}_{ij}f^{da}_{jk}f^{da}_{ki})\\
& = & \delta^1(f_{ij})_{abcd}\delta^1(f_{jk})_{abcd}\delta^1(f_{ki})_{abcd}
\end{eqnarray*} which by applying \Pref{X3} thrice is equal to $1$. So $\theta \in Z^1(X)$.
Therefore, by \Tref{C1t}, there are $\theta' \in \mu_2$ and $\beta \in C^0(X)$ such that
\begin{equation}\label{ttt}
\theta^{ab} = \theta' \beta_a \beta_b
\end{equation}
for all $a$ and $b$.
Substituting \eq{ttt} in \eq{Z5} we again have
$$\theta' \beta_i \beta_j \alpha^{ij}_a\alpha^{ij}_b = \theta^{ij} \alpha^{ij}_a\alpha^{ij}_b= f_{ab}^{ij} = g_{ibja} = g_{aibj} = f^{ab}_{ij} =
\theta^{ab} \alpha^{ab}_i\alpha^{ab}_j = \theta' \beta_a \beta_b \alpha^{ab}_i\alpha^{ab}_j,$$
so fixing $i = i_0$ we get that
$$\alpha^{ab}_j = \beta_{i_0} \beta_j \beta_a \beta_b \alpha^{ab}_{i_0}\alpha^{i_0j}_a\alpha^{i_0j}_b;$$
substituting this and \eq{ttt} back in \eq{Z5}, we get that
\begin{eqnarray*}
g_{aibj} & = & \theta^{ab} \alpha^{ab}_{i} \alpha^{ab}_j \\
& = & (\theta' \beta_a \beta_b) (\beta_{i_0} \beta_i \beta_a \beta_b \alpha^{ab}_{i_0}\alpha^{i_0i}_a\alpha^{i_0i}_b) (\beta_{i_0} \beta_j \beta_a \beta_b \alpha^{ab}_{i_0}\alpha^{i_0j}_a\alpha^{i_0j}_b) \\
& = & \theta' \cdot \beta_a \beta_b \beta_i \beta_j \cdot
\alpha^{i_0i}_a\alpha^{i_0i}_b
\alpha^{i_0j}_a\alpha^{i_0j}_b.
\end{eqnarray*}
Namely, defining $p \in \Cvec$ by $p_{ck} = \alpha_{c}^{i_0k}$,
\begin{equation}\label{gggg}
g = \theta' \cdot \delta^{02}(\beta) \cdot \vdel^1p.
\end{equation}
This shows that in fact $\vdel^1p$ is a well-defined element of $C^2(X)$, proving by \Rref{62} that~$p \in \CvecI$ and $g \in \pm \Im(\delta^{02})B^2(\vec{X})$.
\end{proof}

\begin{prop}\label{pm7}
$Z^2(X) \cap \Im(\delta^{02}) \sub B^2(\vec{X})$.
\end{prop}
\begin{proof}
Let $\alpha \in C^0(X)$, and assume $\delta^{02}\alpha \in Z^2(X)$.
By \eq{03use}, $\delta^{03}\alpha = \delta^2\delta^{02}\alpha = \one$.
Applying this equality to arbitrary pairs of $3$-cells with $7$ joint vertices, we conclude that $\alpha$ is a constant, and then $\delta^{02}\alpha = \alpha^4 = \one$.
 \end{proof}

\begin{cor}\label{Z2B2}
$Z^2(X) = \pm B^2(\vec{X})$ and $Z^2(X)/B^2(\vec{X}) 
\isom \mu_2$.
\end{cor}
\begin{proof}
Recall that the lattice of subgroups in an abelian group is modular. Notice that $-\one \in Z^2(X)$. Now
\begin{eqnarray*}
Z^2(X) & = & Z^2(X) \cap (Z^2(X)\Im(\delta^{02})) \\
& \stackrel{\rm{Thm}~\ref{goodC2}}{=} & Z^2(X) \cap (\pm B^2(\vec{X})\Im(\delta^{02})) \\
& = & \pm [Z^2(X) \cap (B^2(\vec{X})\Im(\delta^{02}))] \\
& \stackrel{\rm{modularity}}{=} & \pm [(Z^2(X) \cap \Im(\delta^{02}))B^2(\vec{X})] \\
& \stackrel{{\rm{Prop}}~\ref{pm7}}{=} & \pm B^2(\vec{X}).
\end{eqnarray*}
It remains to show that $-\one \not \in B^2(\vec{X})$. Otherwise, $-\one = \vdel^1f$ for some $f \in \CvecI$. Let $a,b,i,j,k$ be distinct vertices, and consider the three $2$-cells $(aibj), (aibk), (ajbk)$: by assumption we have that
$$-1 = f_{ai}f_{aj}f_{bi}f_{bj} = f_{ai}f_{ak}f_{bi}f_{bk} = f_{aj}f_{ak}f_{bj}f_{bk},$$
but multiplication results in a contradiction.
\end{proof}

\subsection{$B^2(\vec{X})$ and $B^2(X)$}\label{ss:[-1]}
Fix a linear ordering $<$ of the vertices. Let $[-1] \in C^2(X)$ be the function defined for a $2$-cell~$c$ by
$$\begin{cases}
[-1]_c = +1 & \mbox{if the 
vertices of $c$ can be read in increasing order},
\\
[-1]_c = -1 & \mbox{otherwise}.
\end{cases}$$
We also tautologically set $[+1]_c = +1$.

Let $\psi \in \CvecI$ be the {\bf{order function}} associated to~$<$, defined by $\psi_{ij}= +1$ if $i<j$ and $\psi_{ij} = -1$ otherwise. Clearly $N\psi = -\one$.

\begin{rem}\label{thisone}
$\vdel^1\psi = -[-1]$. The diagram below depicts the three possible orderings of the vertices of a $2$-cell, with the values of $\psi$ denoted on the edges and the value of $\delta^1\psi$ circled in the center, indeed being opposite to the respective value of $[-1]$.
$$
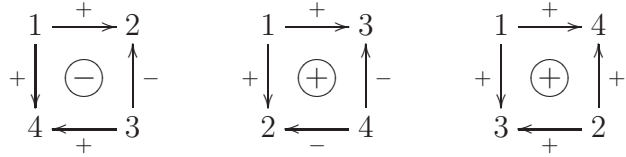
\xymatrix{
1 \ar@{->}[r]^+ \ar@{->}[d]_+ \ar@{}[rd]|*+[o][F-]{-} & 2 \ar@{<-}[d]^{-}\\
4 \ar@{<-}[r]_{+} &  3  \\
} \qquad \xymatrix{
1 \ar@{->}[r]^+ \ar@{->}[d]_+ \ar@{}[rd]|*+[o][F-]{+} & 3 \ar@{<-}[d]^{-}\\
2 \ar@{<-}[r]_{-} & 4 \\
} \qquad \xymatrix{
1 \ar@{->}[r]^+ \ar@{->}[d]_+ \ar@{}[rd]|*+[o][F-]{+} & 4 \ar@{<-}[d]^{+}\\
3 \ar@{<-}[r]_{+} & 2 \\
}$$
\end{rem}

\begin{thm}\label{H2=}
The second cohomology of $X$ is  $H^2(X) \isom \mu_2 \times \mu_2$.
Explicitly, $Z^2(X) = \sg{-\one, [-1], B^2(X)}$ and
$B^2(\vec{X}) = \sg{-[-1],B^2(X)}$.
\end{thm}
\begin{proof}
We first show that $B^2(\vec{X})/B^2(X) \isom \mu_2$. The argument will be easier to follow using Figure~\ref{fig1}.
By definition of $\CvecII$, the induced norm map $$N \,\co\, \CvecI/\CvecII \,\lra\, Z^1(X)/B^1(X)$$ is a well-defined embedding into $H^1(X) = \mu_2$ (\Cref{H1=mu2}).
Similarly, by the definition of $B^2(\vec{X})$ and \Pref{deltaC1''}, $\vdel^1$ induces an isomorphism from $\CvecI/\CvecII$ to $B^2(\vec{X})/B^2(X)$. To conclude the proof, we will show that
$\CvecI \neq \CvecII$. We have that $-\one \not \in B^1(X)$ because $\Delta(-\one) = -1$; so since $\psi \in \CvecI$ chosen above satisfies $N\psi = -\one$, we conclude that $\psi \not \in \CvecII$. It follows that $\vdel^1\psi = -[-1]$ generates $B^2(\vec{X})/B^2(X) \isom \mu_2$.

By \Cref{Z2B2}, $Z^2(X) = \pm B^2(\vec{X}) = (\pm \one) [\pm 1]B^2(X)$, and the index $\dimcol{Z^2(X)}{B^2(X)}$ is equal to~$4$, but the quotient is a group of exponent~$2$, so it equals $\mu_2 \times \mu_2$.
\end{proof}

\begin{cor}\label{finZ2}
Let $g \in Z^2(X)$. Then there are unique $\theta,\pi \in \mu_2$, and some $f \in B^2(X)$, such that
\begin{equation}\label{finZ2:}
g = \theta \cdot [\pi] \cdot \delta^1 f.
\end{equation}
\end{cor}

\begin{rem}
Let $<$ and $<'$ be two linear orders on the set of vertices. Let $[-1]$ and $[-1]'$ be the corresponding functions as defined above. By \Tref{H2=},
$[-1][-1]' \in B^2(X)$, which we now demonstrate explicitly. Let $\psi$ and $\psi'$ be the order functions associated to the order relations. Since $N\psi = N\psi' = -\one$, $N(\psi\psi') = \one$, so that $\psi\psi' \in C^1(X)$, and as computed in \Rref{thisone}, $\delta^1(\psi\psi') = [-1][-1]'$.
\end{rem}

\subsection{Detecting maps}\label{ss:Delta}

We define $\Delta', \Delta'' \co Z^2(X) \ra \mu_2$ by setting
$$\Delta'(g) = g_{ijk\ell}\, g_{ikj\ell}\, g_{ij\ell k}, \qquad \mbox{and} \qquad \Delta''(g) = g_{aibj}\,g_{ajbk}\,g_{akbi};$$
where the vertices are arbitrary.

\begin{prop}\label{const'''}
The maps $\Delta'$ and $\Delta''$ are well-defined on $Z^2(X)$.
\end{prop}
\begin{proof}
Applying \Cref{finZ2}, we need to verify the claim for three types of functions.
\begin{enumerate}
\item Since each formula involves three entries of the function, $\Delta'(-\one) = \Delta''(-\one) = -1$ are well-defined.
\item The consecutive indices involved in $\Delta'$ and $\Delta''$ cover each pair in the graphs ($K_4$ and $K_{2,3}$, respectively) twice, hence $\Delta'(\delta^1f) = \Delta''(\delta^1f) = 1$ are well-defined.
\item Now consider $g = [-1]$. Since $\Delta'(g)$ is the product of the three possible orderings of the vertices of a square, we have that $\Delta'([-1]) = (+1)(-1)(-1) = 1$. Now consider $\Delta''([-1])$. There are $5!$ ways to order the indices $a,b,i,j,k$, but only $\binom{5}{2} = 10$ up to symmetry of the graph $K_{2,3}$. We note that $[-1]_{aibj} = +1$ if and only if the arcs connecting $i$ with~$j$ and $a$ with~$b$ through the upper half plane intersect. It is now easy to see that $[-1]_{aibj}$, $[-1]_{ajbk}$ and $[-1]_{akbi}$ are all $-1$ if $a,b$ are consecutive or if $i,j,k$ are consecutive; and that they are equal to $+1,+1,-1$ otherwise. In both cases the product is $-1$, so $\Delta''([-1]) = -1$ is well-defined.
\end{enumerate}
\end{proof}

More explicitly, we have
$$\Delta'(-\one) = -1, \quad \Delta'([-1]) = +1, \quad \Delta'(\delta^1f) = +1;$$
$$\Delta''(-\one) = -1, \quad \Delta''([-1]) = -1, \quad \Delta''(\delta^1f) = +1$$
for every $f \in C^1(X)$.

\begin{cor}\label{abouthere}
For $g \in Z^2(X)$,
\begin{enumerate}
\item $g \in B^2(\vec{X})$ if and only if $\Delta''(g) = 1$, and
\item $g \in B^2(X)$ if and only if $\Delta'(g) = \Delta''(g) = 1$.
\end{enumerate}
Moreover, in the presentation \eq{finZ2:}, $\theta = \Delta'(g)$ and $\pi = \Delta'(g)\Delta''(g)$.
\end{cor}

Although unnecessary in this section, we record an explicit proof for the fact that $\Delta'$ is well-defined on $Z^2(X)$.
\begin{prop}\label{13.9o}
Let $g \in C^2(X)$. Let $i,j,k,\ell$ and $i',j',k',\ell'$ be eight distinct vertices. Then
\begin{equation}\label{Thet}
(\Delta'g)_{ijk\ell}(\Delta'g)_{i'j'k'\ell'}
\end{equation}
is a product of three entries of $\delta^2g$.
\end{prop}
\begin{proof}
Take the product of $\delta^2g$ over the three cubes on the vertices $i,j,k,\ell,i',j',k',\ell'$ depicted below. The ``side'' faces cancel, and only the product of the top and bottom faces remain, which is equal to $+1$ by assumption.
$$
\CUBE{i}{j}{k}{\ell}{i'}{j'}{k'}{\ell'}{{1}{3}{4}{6}}
\qquad
\CUBE{i}{k}{j}{\ell}{i'}{k'}{j'}{\ell'}{{2}{3}{5}{6}}
\qquad
\CUBE{i}{j}{\ell}{k}{i'}{j'}{\ell'}{k'}{{1}{5}{4}{2}}
$$
\end{proof}

\begin{rem}\label{T8}
$\Delta' \circ \vdel^1 = \Delta \circ N$ on $\CvecI$. Indeed, for $f \in \CvecI$ and arbitrary vertices $i,j,k,\ell$,
\begin{eqnarray*}
\Delta'(\vdel^1f) & =&  (f_{ij}f_{i\ell}f_{kj}f_{k\ell})(f_{ik}f_{i\ell}f_{jk}f_{j\ell})(f_{ij}f_{ik}f_{\ell j}f_{\ell k}) \\
& = & (f_{kj}f_{jk})(f_{k\ell}f_{\ell k})(f_{j\ell}f_{\ell j}) = \Delta(Nf).
\end{eqnarray*}
\end{rem}

\section{Similarity of functions}\label{sec:prel}

The elementary observations of this section will be repeatedly used in the testability proofs in the coming sections.  
We adopt the following notation, motivated by topological uniformity. Recall \Dref{Xk} for~$X^{[k]}$ and $\Fun{k}$.
\begin{defn}\label{simnot}
Let $f,f' \in \Fun{k}$. We write $f \sim_p f'$ if $$\norm{ff'} = \Pr{f_x \neq f'_{x}} \leq p,$$
where the probability is taken by letting the vector $x \in X^{[k]}$ be uniformly random.  The same notation is used for 
functions in $C^d(X)$ and function $X^{[k]} \times X^{[k]} \ra \mu_2$.
\end{defn}
We freely use the facts that $f\sim_p f'$ if and only if $ff' \sim_p \one$, and that $f\sim_p f' \sim_{p'} f''$ implies $f \sim_{p+p'} f''$.

\begin{lem}\label{half}
Let $f \co X^{[k]} \ra \mu_2$. Let $f\times f$ be the function $X^{[k]}\times X^{[k]} \ra \mu_2$ defined by $(f\times f)_{\alpha,\beta} = f_\alpha f_\beta$.
\begin{enumerate}
\item\label{halfii} If $f \sim_{p} \theta$ for a constant $\theta \in \mu_2$, then $f \times f\sim_{2p} \one$.
\item\label{halfipre} If $f\times f \sim_{p'} \one$, then $f \sim_{(\frac{1}{2}+p')p'} \theta$ for some constant $\theta \in \mu_2$.
\item\label{halfi} If $f\times f \sim_{p'} \one$, then $f \sim_{p'} \theta$ for some
constant $\theta \in \mu_2$.
\end{enumerate}
\end{lem}
\begin{proof}
Let $p' = \Pr{(f\times f)_{x,y} \neq 1}$ and $p = \Pr{f_x \neq \theta}$ where $\theta$ is the majority vote on the values of $f$, so that $p \leq \frac{1}{2}$. Since $p' = 2p(1-p)$, we have that $p \leq 
(\frac{1}{2}+p')p' \leq p' \leq 2p$. The fact that $\Pr{f_x \neq \theta} = p \leq (\frac{1}{2}+p')p'$ implies \eq{halfipre}. Since $(\frac{1}{2}+p')p' \leq p'$, $\mbox{\eq{halfipre}} \Rightarrow \mbox{\eq{halfi}}$. Finally $p'\leq 2p$ implies \eq{halfii}.
\end{proof}

Every formula of the form (say) $g_{ijk} = (\delta^0 \alpha)_{ij} (\delta^0 \alpha)_{jk}$ proves that if $\delta^0\alpha = \one$ then $g = \one$. This formula also shows that if $\delta^0 \alpha \sim_p \one$ then $g \sim_{2p} \one$, which is the kind of argument we will repeatedly need below. Indeed, when $(ijk) \in X^{[3]}$ is uniformly distributed, so are $(ij),(jk) \in X^{[2]}$.

However, if $a$ is a fixed vertex and $g_{ijk} = (\delta^0 \alpha)_{ia} (\delta^0 \alpha)_{ja}$, the first implication remains, but the probabilistic one breaks down, for $(\delta^0 \alpha)_{*a}$ need not be close to $\one \in C^0(X)$ even when $\delta^0 \alpha \sim \one \in C^1(X)$, since errors may congregate around~$a$ (the star notation is explained below). We thus need a way to describe connections of the former type. The name ``formal'' is alluding to both ``formulaic'' (for the defining formula \eq{formu}) and ``formational'' (for the formation $u_1,\dots,u_\ell$ on the given vertices.)

\begin{defn}\label{firstformaldef}
Let $X$ be a (simplicial or cubical) complex.
\begin{enumerate}
\item A function $f \in \Fun{k}$ is {\bf{formal in $g \in \Fun{k'}$}}, of length~$\ell$, if there are vectors $u_1,\dots,u_\ell \in X^{[k']}$ whose vertices are contained in $\set{v_1,\dots,v_k}$, such that
\begin{equation}\label{formu}
f_{\s(v_1),\dots,\s(v_k)} = \prod_{i=1}^{\ell} g_{\s(u_i)}
\end{equation}
for the permutations~$\s$ of the vertex set~$\cells[X]{0}$, extended in the obvious manner to act on all vectors.
\item An operator $\phi \co C^d(X) \ra \Fun{k}$ is {\bf{formal in
$\phi' \co C^d(X) \ra \Fun{k'}$}}, of length~$\ell$, if $\phi f$ is formal in $\phi' f$ via the same formula of length~$\ell$.
\end{enumerate}
\end{defn}

\begin{lem}\label{goon}
Suppose~$f$ is formal of length~$\ell$ in~$g$. If $g \sim_p \one$ then $f \sim_{\ell p} \one$.
\end{lem}
\begin{proof}
For a uniformly random vector $(\s(v_1),\dots,\s(v_k)) \in X^{[k]}$, each $\s(u_i)$ is uniformly random, and therefore $\Pr{g_{\s(u_i)} \neq 1} = p$.
\end{proof}

The notion of formality is mostly suitable for complete complexes. Indeed, for any operator to be formal in $\delta^d$, it has to be implicitly assumed that $X$ is complete in dimension $d+1$ (because the $(d+1)$-cells uniformly participate in the product).

In the next section we will need a probabilistic analog of \Pref{KerDelta}:
\begin{prop}\label{KerDeltaprob}
The differential $\delta^1$ is formal of length $2$ in~$\Delta$.
\end{prop}
\begin{proof}
For every $f \in C^1(X)$,
$(\delta^1f)_{ijk\ell} = (\Delta f)_{ij\ell}(\Delta f)_{jk\ell}$.
\end{proof}

Applying \Cref{goon} to \Pref{KerDeltaprob}  we get:
\begin{cor}
If $\Delta f \sim_p \one$ then $\delta^1f \sim_{2p} \one$.
\end{cor}

\smallskip
We use asterisks to denote entries in a function $f \in \Fun{k}$. Replacing an asterisk by a specific value defines a function in $\Fun{k-1}$. For example, if $f \in \Fun{4}$, then $f_{a***}, f_{*a**} \in \Fun{3}$.

\begin{lem}\label{manycond}
Suppose $f^i \in \Fun{k_i}$ for $i = 1,\dots,N$ are functions such that $f^i \sim_p \one$ for each~$i$, where~$p$ is fixed. Let $s > N$ be a real number. If~$X$ is large enough, then there is a vertex $a \in \cells[X]{0}$ for which $f^i_{*\cdots*a*\cdots*} \sim_{sp}\one$ for each $i$. (Prior to the statement, the fixed vertex can be placed arbitrarily for each~$i$).
\end{lem}
\begin{proof}
For each~$i$, the proportion of $a \in \cells[X]{0}$ for which $f^i_{*\cdots*a*\cdots*} \sim_{sp} \one$ does not hold is at most $s^{-1}$, so the proportion of vertices for which at least one of the conditions fail is at most $Ns^{-1}<1$.
\end{proof}

\section{Testing $B^1(X)$}\label{sec:testB1}

The result below is proved in \cite[Subsection~7.2]{Dinur} by direct probabilistic methods. 
We prove it here in order to demonstrate the usage of $\Delta$, anticipating the more complicated proof in the next section. Let $X$ be a complete $2$-dimensional cubical complex.

Let $p > 0$ be a constant.
\begin{thm}\label{MainC1}
Let $f \in C^1(X)$. If $\delta^1f \sim_p \one$, then there are $\theta \in \mu_2$ and $\alpha \in C^0(X)$ such that $f \sim_{3p} \theta \cdot\delta^0 \alpha$.
(Namely, $f_{ij} = \pm \alpha_i \alpha_j$ with probability of error at most $3p$).
\end{thm}
\begin{proof}
Recall from~\eq{Deltadefall} of \Sref{sec:4} the function $\Delta \co C^1(X) \ra \Fun{3}$ defined by $(\Delta f)_{ijk} = f_{ij}f_{jk}f_{ki}$. In \Lref{bT} we proved that $\Delta f$ is a constant if and only if $f \in Z^1(X)$.
Moreover, if $i,j,k$ and $i',j',k'$ are distinct, the proof of \Lref{bT} shows that
$$(\Delta f)_{ijk} (\Delta f)_{i'j'k'} = (\delta^1 f)_{ijj'i'} (\delta^1 f)_{jkk'j'} (\delta^1f)_{kii'k'},$$
so that $\Delta f \times \Delta f$ is formal of length~$3$ in $\delta^1f$. 
(The case when $\set{i,j,k}$ and $\set{i',j',k'}$ intersect is negligible).
By \Lref{goon}, since $\delta^1f \sim_p \one$, we have that  $\Delta f \times \Delta f \sim_{3p} \one$. By \Lref{half}(\ref{halfi}), $\Delta f \sim_{3p} \theta$ for a constant $\theta$.  Let $f' = \theta f$, so that $\Delta f'  \sim_{3p} \theta^3\theta = \one$.

Choose a vertex $a$ such that $(\Delta f')_{a**} \sim_{3p} \one$ (see \Lref{manycond}).
It follows that $f'_{ij} \sim_{3p} f'_{ia}f'_{ja} = \delta^0(f'_{*a})$, and $f \sim_{3p} \theta \cdot \delta^0(f'_{*a})$.
\end{proof}

In the terminology of \Sref{sec:2}, we proved:
\begin{cor}
The expansion constant of the \frst\ incidence geometry of $X$ (composed of vertices, edges and squares), with respect to the complement
$\sg{-\one}$, is at most $\omega = \frac{1}{3}$.
\end{cor}
\begin{cor}
The space $\pm B^1(X)$ is testable.
(The tester is the function $\delta^1 \co C^1(X) \ra B^2(X)$, and each entry requires $4$ queries).
\end{cor}

\section{Testing $B^2(X)$}\label{sec:testB2}

In this section we prove that $\delta^2$ tests $f \in C^2(X)$ for being in $B^2(X)$.
Following \Ssref{ss:Delta}, let $\Delta' \co C^2(X) \ra \Fun{4}$ and $\Delta'' \co C^2(X) \ra \Fun{5}$ be defined (for arbitrary $g \in C^2(X)$) by
$$(\Delta' g)_{ijk\ell} = g_{ijk\ell}\, g_{ikj\ell}\, g_{ij\ell k}, \qquad (\Delta''g)_{ab;ijk} = g_{aibj}\,g_{ajbk}\,g_{akbi}.$$

\begin{lem}\label{wefindm}
Let $g \in C^2(X)$. If $\delta^2g \sim_p \one$, then
$\Delta'g \sim_{3p} \theta$ and $\Delta''g \sim_{6p} \pi$ for
constants $\theta,\pi \in \mu_2$.
\end{lem}
\begin{proof}
In \Lref{13.9o} we show that
$\Delta' \times \Delta'$ is formal of length~$3$ in~$\delta^2$.
Therefore $\Delta'g \times \Delta'g \sim_{3p} \one$ (\Lref{goon}), so $\Delta'g \sim_{3p} \theta$ by \Lref{half}(\ref{halfi}).

In \Pref{X3} we show that $(\Delta''g)_{ijk}(\Delta''g)_{jk\ell}$ is a product of four entries of $\delta^2g$. By the argument of \Lref{bT} we see that  for distinct $i,j,k,i',j',k'$, $(\Delta''g)_{ijk}(\Delta''g)_{i'j'k'}$ is a product of $3\cdot 4=12$ entries of $\delta^2g$, but since the four $3$-cells participating in the computation in \Pref{X3} only depend on $i,j,k,\ell$ through the same two entries, six of those cancel in pairs, and we get $\Delta''g \times \Delta''g \sim_{6p} \one$. The proof concludes as above.
\end{proof}

We now prove the testability version of \Cref{finZ2}. Let $p > 0$ be a constant.
\begin{thm}\label{MainC2}
Let $g \in C^2(X)$. If $\delta^2g \sim_p \one$, then there are $\theta,\pi \in \mu_2$ and $f \in C^1(X)$ such that $g \sim_{rp} \theta [\pi]\cdot \delta^1f$ for a constant~$r < 1504$.
\end{thm}
\begin{proof}
By \Lref{wefindm}, there are $\theta, \pi \in \mu_2$ such that $\Delta'g \sim_{3p} \theta$ and $\Delta''g \sim_{6p} \pi$. Replacing $g$ by $\theta [\pi] g$ and applying \Cref{abouthere}, we may from now on assume $\Delta'g \sim_{3p} \one$ and $\Delta''g \sim_{6p} \one$.

Fix a real number $s>3$.
By \Lref{manycond} there is a vertex $a_0$ for which
$$(\Delta'g)_{a_0 * * *} \sim_{3sp} \one, \qquad (\Delta'' g)_{a_0 *; ***} \sim_{6sp} \one, \qquad (\Delta'' g)_{**;a_0**} \sim_{6sp} \one.$$
Again by \Lref{manycond}, building on the first two statements, there is a vertex $b_0$ for which
$$(\Delta'g)_{a_0 * b_0 *} \sim_{3s^2p} \one, \qquad (\Delta'' g)_{a_0 b_0; ***} \sim_{6s^2p} \one, \qquad (\Delta'' g)_{a_0 *; b_0**} \sim_{6s^2p}\one.$$
Define $h_{ij} = g_{a_0ib_0j}$, which is symmetric because $g \in C^2(X)$, so that $h \in C^1(X^{a_0b_0})$. Now~$\Delta h  = (\Delta''g)_{a_0b_0;***} \sim_{6s^2p} \one$, so by \Pref{KerDeltaprob} $\delta^1h \sim_{12s^2p} \one$. By \Tref{MainC1}, and using again the fact that $\Delta h \sim_{6s^2p} \one$, there is $\beta \in C^0(X)$ such that $h \sim_{36s^2p} \delta^0\beta$.

We now define $f' \in \Cvec$ by taking $f'_{a_0j} = 1$ for all $j \neq a_0$, $f'_{b_0 j} = \beta_j$ for all $j \neq a_0,b_0$, and
$$f'_{ij} = \beta_i g_{a_0b_0ij}$$
for all $i,j$ disjoint from $a_0,b_0$.

We claim that $f'_{ji} \sim_{39s^2p} f'_{ij}$. Indeed, $$g_{a_0b_0ij} g_{a_0b_0ji} = (\Delta' g)_{a_0b_0ij} g_{a_0 i b_0 j} \sim_{3s^2p} g_{a_0 i b_0 j} = h_{ij},$$
so
$f'_{ij}f'_{ji} = \beta_i \beta_j g_{a_0b_0ij} g_{a_0b_0ji} \sim_{3s^2p} \beta_i \beta_j h_{ij} \sim_{36s^2p} \one$. By keeping the entries where $f'_{ij} = f'_{ji}$ and fixing the value $1$ at the other entries, we obtain $f \in C^1(X)$ such that $f' \sim_{39s^2p} f$.

Using the symmetry of $f$, and applying $\Delta''$ twice, we now have that
\begin{eqnarray*}
(\delta^1f)_{aibj} & = & f_{ai}f_{bi}f_{aj}f_{bj} \\
& \sim_{4\cdot 39s^2p} &  f'_{ai}f'_{bi}f'_{aj}f'_{bj} \\
& = & \beta_a^2\beta_b^2 g_{a_0b_0ai}g_{a_0b_0bi}g_{a_0b_0aj} g_{a_0b_0bj} \\
& = & g_{a_0b_0ai}g_{a_0b_0bi}g_{a_0b_0aj} g_{a_0b_0bj} \\
& = & (\Delta'' g)_{a_0 a; b_0 i j}  (\Delta' g)_{a_0 b; b_0 i j} \cdot g_{a_0iaj} g_{a_0ib_j}\\
& \sim_{9s^2p} &  g_{a_0iaj} g_{a_0ib_j}\\
& = & (\Delta'' g)_{ij; a_0 a b} \cdot g_{aibj}\\
& \sim_{6sp} & g_{aibj},
\end{eqnarray*}
so that $g \sim_{(165s^2+6s)p} \delta^1f$. Taking $s>3$ small enough proves the claim.
\end{proof}

\begin{cor}
The expansion constant of the \scnd\ incidence geometry of the complete $3$-dimensional cubical complex~$X$ (composed of edges, squares and all cubes), with respect to the complement
$\sg{[-1],-\one}$, is at most $\omega = \frac{1}{1504}$.
\end{cor}

\begin{cor}\label{finally1}
Let $X$ be the complete $3$-dimensional cubical complex.
Then the differential $\delta^2 \co C^2(X) \ra C^3(X)$ tests $Z^2(X)$
(each entry of the test requires~$6$ queries).
\end{cor}

We can now give a precise formulation of \Trefs{maincomb}{maincomb2}, based on \Dref{Testdef}. Since $Z^2(X) = \pm B^2(\vec{X}) = \sg{\pm \one, [\pm 1], B^2(X)}$, the proofs follow from \Cref{finally1}.
\begin{cor}\label{2.1forreal}
The differential $\delta^2 \co C^2(X) \ra C^3(X)$ is a $6$-query test on functions $g \in C^2(X)$ for being of the form $g_{ijk\ell} = \pm f_{ij}f_{kj}f_{k\ell}f_{i \ell}$ for some $f \co X \times X \ra \mu_2$. More explicitly, for every $g \in C^2(X)$ there is $f$ such that $\|\pm \vec{\delta}^1f \cdot g\| \leq 1504 \|\delta^2g\|$.
\end{cor}

\begin{cor}\label{2.2forreal}
The differential $\delta^2 \co C^2(X) \ra C^3(X)$ is a $6$-query test on functions $g \in C^2(X)$ for being of the form $g_{ijk\ell} = \pm [\pm 1] f_{ij}f_{kj}f_{k\ell}f_{i \ell}$ for some (symmetric) $f \in C^1(X)$. More explicitly, for every $g \in C^2(X)$ there is a symmetric $f$ such that $\|\pm [\pm 1] \delta^1f \cdot g\| \leq 1504 \|\delta^2g\|$.
\end{cor}

\section{Proving testability in general}\label{sec:9}

The explicit constant in \Tref{MainC2} relies on \Lref{wefindm}, which requires combinatorial analysis  special to that particular case. A soft version, without an explicit constant, can be proved through a lemma on formal functions (\Dref{firstformaldef}).

\begin{lem}\label{L1}
Assume that $\phi \co C^d(X) \ra \Fun{k}$ is formal in the identity operator $C^d(X) \ra C^d(X)$.
Assume $\phi f = \one$ for every $f \in Z^d(X)$.
Then $\phi$ is formal in $\delta^d$.
\end{lem}
\begin{proof}
Let $v_1,\dots,v_k$ be the vertices from \Dref{firstformaldef}.
Write $v = (v_1,\dots,v_k)$.

Let $\phi_v \co C^d(X) \ra \mu_2$ be the function defined by $\phi_vf = (\phi f)_v$. View $C^d(X)$ as a vector space over the field of two elements, and let $V^*$ be the subspace of the dual space of $C^d(X)$ spanned by the functionals
$f \mapsto (\delta^df)_c$, where $c$ ranges over the $d$-cells of $X$. By definition $\psi \in V^*$ if and only if $\psi f= 1$ for every
$f \in Z^d(X)$. Therefore, by assumption, $\phi_v \in V^*$. It follows that $\phi_v$ is a product of, say, $m$~entries of $\delta^d(\cdot)$. The desired expression is obtained by permuting the vertices.
\end{proof}

We now outline a proof for testability in complexes of higher dimension.
We say that a system of homomorphisms $\Delta_i \co C^d(X) \ra \Fun{k_i}$ ($i = 1,\dots,u$) {\bf{induce a map from $H^d(X)$}}, if each $\Delta_i g$ is a constant function for $g \in Z^d(X)$, and this constant function is $\one$ for $g \in B^d(X)$. Indeed in this case we obtain a
homomorphism $\tilde{\Delta} \co H^d(X) \ra H = (\mu_2)^{u}$.
\begin{thm}
Suppose there are formal functions $\Delta_i \co C^d(X) \ra \Fun{k_i}$
inducing an isomorphism $\tilde{\Delta} \co H^d(X) \ra H$,
and a map $\nabla \co H \ra Z^d(X)$ such that $\nabla \circ \tilde{\Delta}$ splits the short exact sequence
$$\xymatrix@C=10pt{1 \ar@{->}[rr] & & B^d(X) \ar@{^(->}[rr] & & Z^d(X) \ar@<0pt>@{->}[rr]|{\theta} &&  H^d(X) \ar@{->}[dl]|{\,\tilde{\Delta}}
\ar@{->}[rr] && 1, \\
{} && {} && {} & H \ar@{..>}[ul]|{\nabla} 
& &
}$$
in a way that for some constant~$r$, if $\delta^d g \sim_p \one$ and $\Delta_i g \sim_p \one$ for every~$i$ then there is $f \in C^{d-1}(X)$ such that $g \sim_{rp} \delta^{d-1}f$.
Then $B^d(X) \cdot \Delta(H) \leq C^d(X)$ is testable (with respect to the constant $\omega = r^{-1}$).
\end{thm}
\begin{proof}
Let $g \in C^d(X)$ be a function satisfying $\delta^d g \sim_p \one$.
By assumption each $\Delta_i$ is formal, and constant on $Z^d(X)$.
Therefore $\Delta_i \times \Delta_i$ satisfies the conditions of \Lref{L1}, where the entries of $\delta^d$ are uniformly random when the permutation is applied. So for suitable~$m$,  
$\Delta_ig \times \Delta_ig \sim_{mp} \one$. By \Lref{half}(\ref{halfi}), there are constants $\theta_i$ such that $\Delta_ig \sim_{mp} \theta_i$. Replacing $g$ by $g\cdot \nabla(\theta_1,\dots,\theta_u)$, we obtain a function satisfying $\delta^d g \sim_p \one$ and all the conditions $\Delta_i g \sim_{mp} \one$. By assumption, there is now $f \in C^{d-1}(X)$ such that $g \sim_{rp} \delta^{d-1}f$ for a constant~$r$.
\end{proof}

This is the method proving \Trefs{MainC1}{MainC2}.

\end{document}